\documentclass[pdflatex,sn-mathphys-num]{sn-jnl}

\usepackage{graphicx}%
\usepackage{multirow}%
\usepackage{amsmath,amssymb,amsfonts}%
\usepackage{amsthm}%
\usepackage{mathrsfs}%
\usepackage[title]{appendix}%
\usepackage{xcolor}%
\usepackage{textcomp}%
\usepackage{manyfoot}%
\usepackage{booktabs}%
\usepackage{algorithm}%
\usepackage{algorithmicx}%
\usepackage{algpseudocode}%
\usepackage{listings}%

\theoremstyle{thmstyleone}%
\newtheorem{theorem}{Theorem}
\newtheorem{proposition}[theorem]{Proposition}%
\newtheorem{lemma}[theorem]{Lemma}%
\newtheorem{fact}{Fact}%

\theoremstyle{thmstyletwo}%
\newtheorem{example}{Example}%
\newtheorem{remark}{Remark}%

\theoremstyle{thmstylethree}%
\newtheorem{definition}{Definition}%

\DeclareMathOperator{\PSpec}{PSpec}
\newcommand{\id}{\mathrm{id}}
\newcommand{\A}{\mathcal{A}}
\newcommand{\B}{\mathcal{B}}
\newcommand{\G}{\mathcal{G}}
\newcommand{\C}{\mathbf{C}}
\newcommand{\I}{\mathbf{I}}
\newcommand{\midd}{\; \middle| \;}

\usepackage{enumerate}
\usepackage{faktor}

\begin{document}

\title[On semigroups and groupoids with minimal probabilistic spectrum]{On semigroups and groupoids with minimal probabilistic spectrum}


\author*[ ]{\fnm{Carles} \sur{Card\'o}}\email{ccardo@uic.es}



\affil*[ ]{\orgdiv{Departament de Medicina, Àrea d’Estadística, Salut Pública i
Epidemiologia}, \orgname{Universitat Internacional de Catalunya}, \orgaddress{\street{Josep Trueta s/n}, \city{Sant Cugat del Vallès}, \postcode{08195}, \state{Catalunya}, \country{Spain}}}




\abstract{The equational probabilistic spectrum of a finite algebra is the set of probabilities with which equations are satisfied in the algebra. We study algebras with minimal spectrum, that is, spectra consisting only of the values $1$ and $1/|\A|$. 
We show that, apart from trivial cases, groupoids with minimal spectrum are quasigroups. We further prove that several weak associativity conditions collapse into full associativity, and hence into group structure. 
Finally, we obtain a complete classification of semigroups with minimal spectrum.}





\maketitle

\section{Introduction}\label{sec1}

Probabilistic methods have long been used to study structural properties of algebraic structures. 
A classical example in group theory concerns the probability that two randomly chosen elements 
of a finite group $G$ commute. This probability equals $k(G)/|G|$, where $k(G)$ denotes the 
number of conjugacy classes of $G$ \cite{Gustafson1973, MacHale1974}. In particular, if $G$ 
is non-abelian, this probability is at most $5/8$. Probabilistic and asymptotic counting 
techniques were introduced into group theory by P. Erd\H{o}s and P. Turán in the 1960s 
\cite{ET1, ET2, ET3, ET4, ET6, ET7}.

Extensions of this notion to infinite groups, semigroups, and other algebraic structures 
have been studied by various authors \cite{Joseph1969, Laffey1976, Joseph1977, Eberhard2015, 
Ponomarenko2012, Dixon2002, Abert2003}. In the framework of universal algebra, equational probability can be defined as the proportion of tuples that satisfy an equation in a finite algebra. 
If we fix a finite algebra $\A$ and consider all its equations, we obtain its probabilistic spectrum 
$\PSpec(\A)$. In some algebras this spectrum is dense in the unit interval, while in others it is minimal, consisting only of the values $1$ and $1/|\A|$.

In a complementary work \cite{Cardo2026}, we study the relation between the spectrum and 
primality, where the spectrum typically turns out to be dense. In the present paper we 
focus on the opposite situation, namely when the spectrum is minimal, with particular 
attention to groupoids and semigroups. We prove that, apart from trivial cases, such 
groupoids are quasigroups, and we obtain a complete classification of semigroups with 
minimal spectrum. We also show that several weak forms of associativity collapse into full associativity under this condition. These results show that minimality of the spectrum imposes strong structural constraints on algebraic structures.

\section{The equational probabilistic spectrum of an algebra} \label{EspectreProbabilistic}

We denote an algebra by $\A=(A, F)$ where $A$ is the universe of the algebra and $F$ is a set of basic operations. All algebras considered here are finite. For definitions of common algebras and basic terminology, we refer to \cite{Burris2012,Howie1995}. The remaining notions will be introduced as needed throughout the text.

By an \emph{equation} we mean a pair of terms $(t,t')$, which we write as $t\approx t'$. The number of variables in an equation is the number of distinct variables involved. Given an equation $t\approx t'$ with $k$ variables in an algebra $\A$, we write the set of solutions as
$$\{t \approx t'\}_{\A}= \{ \vec{x} \in A^k \mid t(\vec{x})=t'(\vec{x})\}.$$

\begin{definition} We define the \emph{probability of the equation $t\approx t'$ in the finite algebra $\A$} as the fraction
$$\Pr (t\approx t' \mid \A)=\frac{|\{t \approx t'\}_\A|}{|A|^k}.$$
We call the following set the \emph{equational probability spectrum} of the algebra $\A$
$$\PSpec(\A)=\{ \Pr (t\approx t' \mid \A) \mid t\approx t' \mbox{ is an equation }\}.$$
\end{definition}

We will only deal with finite algebras, so the previous quotient always makes sense. Consider, for example, the probability that two elements commute within the dihedral group $\mathcal{D}_4$ of order eight,
$$\Pr(xy\approx yx \mid \mathcal{D}_4)=\frac{5}{8}.$$
As another example, the probability that two subsets of a set $U$ with $m$ elements have empty intersection is
$$\Pr(x \cap y \approx \emptyset \mid 2^U)=\left(\frac{3}{4}\right)^m.$$

The following property motivates the present article.
\begin{fact} The spectrum of an algebra $\A$ always contains the two values $1$ and $1/|A|$.
\end{fact}
\begin{proof} Indeed, $\Pr(x\approx x \mid \A)=1$ and $\Pr(x\approx y \mid \A)=1/|A|$.
\end{proof}

Many algebras exhibit only these two probability values. We will say that they are algebras with \emph{minimal spectrum}. We will investigate this question in the case of groupoids, but some results can be generalized naturally to other signatures.

\begin{example} The following are examples of well-known probabilistic spectra. For details of the calculations, see \cite{Cardo2026}.

\begin{enumerate}[(i)]
\item Let $\mathcal{P}_n=(\{1, \ldots, n\}, *)$ equipped with the projection operation $x*y=x$.
$$\PSpec(\mathcal{P}_n)=\left\{ \frac{1}{n},1 \right\}.$$

\item Let $\mathcal{S}_2= (\{0,1\}, |)$ equipped with the Sheffer bar operation $|$. Due to the functional completeness of this operation, we obtain:
$$ \PSpec(\mathcal{S}_2) =\left\{\, \frac{s}{2^k} \midd 0 \le s \le 2^k,\; k \ge 0 \,\right\}.$$

\item More generally, for every primal algebra, the spectrum is the set of $n$-adic numbers, where $n$ is the order of the algebra.

\item Let $\mathcal{C}_2=(\{0,1\}, \wedge)$ be the semilattice, where $\wedge$ is the ``and'' operation.
$$\PSpec(\mathcal{C}_2)=\left\{1-\frac{2^p+2^q-2}{2^{p+q+r}} \midd p,q,r \geq 0\right\}.$$

\item We will see later in Lemma~\ref{LemmaGrupAb3} that if $p$ is a prime number
$$\PSpec(\mathbb{Z}_p)=\left\{ \frac{1}{p},1 \right\},$$
and hence the spectrum is minimal.
\end{enumerate}

\end{example}

We will also need the following basic properties. An algebra $\A$ is said to be \emph{derived from the algebra $\B$} if both have the same universe and the operations of $\A$ are terms of $\B$. We say that $\A$ is a \emph{reduct} of $\B$ if $\A$ is obtained by removing some operations from the signature of $\B$.

\begin{proposition} Let $\A$ and $\B$ be algebras.
\begin{enumerate}[(i)]
\item $\Pr( t\approx t' \mid \A\times \B)=\Pr( t\approx t' \mid \A)\cdot \Pr( t\approx t' \mid \B)$, where both algebras have the same signature.
\item $\PSpec(\A \times \B) \subseteq\PSpec(\A)\cdot \PSpec(\B)$, where the product of sets is taken pointwise and where the two algebras have the same signature.
\item $\PSpec(\A^m)=\{\alpha^m \mid \alpha \in \PSpec(\A)\}$, for any $m\geq 1$.
\item If $\A$ is derived from $\B$, then $\PSpec (\A) \subseteq \PSpec (\B)$.
\item If $\A$ is a reduct of $\B$, then $\PSpec (\A) \subseteq \PSpec (\B)$.
\item Two isomorphic or anti-isomorphic groupoids have the same spectrum.
\end{enumerate}
\end{proposition}
\begin{proof}\begin{enumerate}[(i)]
\item We note that the set $\{t\approx t'\}_\A\times \{t\approx t'\}_\B$ and the set $\{t\approx t'\}_{\A\times \B}$ are in bijection. So,
\begin{align*}
\Pr( t\approx t' \mid \A\times \B) &=\frac{|\{t\approx t'\}_{\A\times \B}|}{|A\times B|^k}=\frac{|\{t\approx t'\}_\A| \cdot | \{t\approx t'\}_\B| }{|A|^k\cdot|B|^k}\\
&=\Pr( t\approx t' \mid \A)\cdot \Pr( t\approx t' \mid \B)
\end{align*}
\item Each probability of the spectrum of $\A\times \B$ comes from an equation, and therefore by (i), it can be decomposed as the product of a probability of the spectrum of $\A$ and one of those of $\B$.
\item The same reasoning as in (ii), but since the algebras are equal and the same equation is considered, $\alpha$ is the value of a probability in $\A$ if and only if $\alpha^m$ is for $\A^m$.
\item Any equation of $\A$ can be expressed in terms of an equation of $\B$.
\item Since $\A$ is a reduct of $\B$, every equation of $\A$ is an equation of $\B$.
\item If two groupoids are isomorphic, the result follows immediately. Suppose they are anti-isomorphic. Given a term $t$, we can define its \emph{anti-term} $\bar{t}$, reversing the order of the operands. If $\bar{\A}$ is an anti-isomorphic groupoid to $\A$, then $\Pr(t\approx t'\mid \A)=\Pr(\bar{t}\approx \bar{t}'\mid \bar{\A})$. Since there is a bijection between the equations of $\A$ and those of $\bar{\A}$, we have that $\PSpec(\A)=\PSpec(\bar{\A})$.
\end{enumerate}
\end{proof}

\begin{remark} Note that if the semigroup $\mathcal{G}'=(G, \cdot)$ is the reduct of a group $\mathcal{G}=(G, 1,(\cdot)^{-1},\cdot)$, then their spectra are equal.
On the one hand, we clearly have $\PSpec(\mathcal{G}') \subseteq \PSpec(\mathcal{G})$. On the other hand, if $w \approx w'$ is an equation in the group signature, we can eliminate any occurrence of the unit element $1$, unless $w=1$ or $w'=1$. In this case, we introduce a new variable $y$ and replace the equation by $wy\approx w'y$. Since we only consider finite groups, we can replace each inverse $x^{-1}$ by $x^{n-1}$, where $n$ is the order of $\G$.

Therefore, when dealing with semigroups arising from groups, it suffices to consider equations of the form $t\approx 1$ without loss of generality.
\end{remark}

\section{Minimality of the spectrum and quasigroups} \label{MinSpec}
An \emph{identity} is an equation that holds universally in an algebra, that is, $\mathcal{A}\models \forall \vec{x}\,\, (t\approx t')$. A \emph{quasi-identity} is a formula of the form $\forall \vec{x}\,\, (t\approx t' \rightarrow s\approx s')$.

Given an equation $t\approx t'$ involving at least the variables $x$ and $y$, 
the equation resulting from identifying $x$ with $y$ is called a \emph{specialization} and we denote it by
$$\faktor{t\approx t'}{x\approx y}\,\,.$$
For example, in a groupoid we can write
$$\faktor{(xy\approx zy)}{(x\approx y)} =( x^2\approx zx).$$
The following Lemma is crucial for the rest of the article.
\begin{lemma} \label{LemmaIdentEq} Let $\A$ be an algebra with minimal spectrum. Suppose that $\A$ satisfies the identity $\faktor{t\approx t'}{x\approx y}$. Then, either
\begin{enumerate}[(i)]
\item $\A$ satisfies the identity $t\approx t'$
\item or $\A$ satisfies the quasi-identity $t\approx t' \rightarrow x\approx y$.
\end{enumerate}
\end{lemma}
\begin{proof} Let $k$ be the number of variables of $t\approx t'$ and $n$ the order of the algebra $\A$. Let us assume, without loss of generality, that $x=x_{k-1}$ and $y=x_k$; otherwise, we simply rearrange the components. First, we note that there is an injective mapping,
\begin{align*}
\left\{ (x_1, \ldots, x_{k-1}) \in A^{k-1}\midd\faktor{t\approx t'}{x\approx y} \right\} &\longrightarrow \{ (x_1, \ldots, x_{k-1}, x_k)\in A^k\mid t\approx t'\}\\
(x_1, \ldots, x_{k-1}) & \longmapsto (x_1, \ldots , x_{k-1}, x_{k-1})
\end{align*}
On the one hand, this function is well defined, since the solutions of the equation $\faktor{t\approx t'}{x\approx y}$ are precisely the solutions of $t\approx t'$ that satisfy $x=y$, that is, $x_{k-1}=x_k$. On the other hand, it is clear that the function is injective. Hence, we obtain the inequality:
$$|\{ t\approx t'\}_\A| \geq \left| \left\{ \faktor{t\approx t'}{x\approx y}\right\}_\A\right|. $$
Suppose that (ii) does not hold. Then (i) must hold. Suppose that (ii) does not hold. Then there exists a tuple $\vec{a}\in A^k$ that satisfies $t\approx t'$, but such that its component $x$ and its component $y$ are different. Therefore, we have that
$$|\{ t\approx t'\}_\A| > \left| \left\{ \faktor{t\approx t'}{x\approx y}\right\}_\A\right|. $$
We divide by $n^k$ and we obtain
$$\frac{|\{ t\approx t'\}_\A|}{n^k}> \frac{\left| \left\{\faktor{t\approx t'}{x\approx y}\right\}_\A\right|}{n^k}=\frac{1}{n}\frac{\left| \left\{ \faktor{t\approx t'}{x\approx y}\right\}_\A\right|}{n^{k-1}}. $$
Therefore,
$$\Pr(t\approx t' \mid \A)> \frac{1}{n} \Pr\left(\faktor{t\approx t'}{x\approx y} \midd \A\right).$$ The algebra  $\A$ satisfies universally the identity $\faktor{t\approx t'}{x\approx y} $, that is,
$$\Pr\left(\faktor{t\approx t'}{x\approx y} \midd \A\right)=1.$$ 
Thus, 
$$\Pr(t\approx t' \mid \A)> \frac{1}{n}.$$
However, since $\A$ has minimal spectrum, the probability values must be either $1/n$ or $1$. Therefore,
$$\Pr(t\approx t' \mid \A)=1.$$
In other words, $\A$ satisfies universally $t\approx t'$. 
\end{proof}
  
In group theory, the \emph{cycle type} of a permutation is the vector of the lengths of each cycle into which it decomposes. We say that a permutation is \emph{isocyclic} if its cycle type is $(1, a,\ldots, a)$. Note that there is at least one fixed point. In fact, a more accurate name would be quasi-isocyclic. However, for brevity, we will call it isocyclic. We say that a mono-unary algebra is isocyclic if its operation is an isocyclic permutation.

\begin{theorem}\label{TheoMonoUnary} A monounary algebra has minimal spectrum if and only if it is constant or isocyclic. Two isocyclic monounary algebras are isomorphic if and only if they have the same cycle type.
\end{theorem}
\begin{proof} Consider the first statement. $(\Rightarrow)$ Let $\mathcal{A}=(A,f)$ be a monounary algebra of order $n$. We will assume that $n>1$, since the case $n=1$ is trivial. Consider the equation $f(x)\approx f(y)$. We have that
$$\faktor{(f(x) \approx f(y))}{(x\approx y)}=(f(x)\approx f(x)),$$
which is an identity that is always satisfied in any monounary algebra. By applying Lemma~\ref{LemmaIdentEq}, either $f(x)=f(y)$ for all $x,y\in A$, or if $f(x)=f(y)$, then $x=y$. In the first case, $f$ must necessarily be constant. In the second case, $f$ must be injective. Since $\A$ is finite, injectivity implies bijectivity.
Moreover, the number of fixed points of $f$ must be either $1$ or $n$. This follows from the fact that the probability $\Pr(f(x)\approx x \mid \A)=F/n$, where $F$ is the number of fixed points. Since the probability must be $1$ or $1/n$, we have that $F=1$ or $F=n$.

Now suppose that $n=2$. There is at least one fixed point, and therefore $f$ must be the identity. Therefore suppose that $n\geq 3$ and that $f$ decomposes into at least three cycles of length $1$, $a$ and $b$ with $1< a <b$, which we denote by $C_1$, $C_a$ and $C_b$. Consider the equation $f^a(x)\approx x$. We have that the number of solutions is greater or equal to $1+a$: one solution is the fixed point of $C_1$, and the remaining $a$ solutions correspond to the elements of the cycle $C_a$. On the other hand, if $b>a$, no point of the cycle $C_b$ can be a solution of the equation. Thus, the number of solutions of $f^a(x)\approx x$ is smaller than or equal to $n-b$. This means that the probability of this equation is bounded as
$$\frac{1}{n}<\frac{1+a}{n}\leq \Pr(f^a(x) \approx x \mid \A)\leq \frac{n-b}{n}< 1.$$
By the minimality of the spectrum, this is impossible. Therefore, all the cycles into which $f$ decomposes and that are not fixed points must have the same length.

$(\Leftarrow )$ It is easy to see that if $f$ is constant it has minimal spectrum. Suppose $f$ is a permutation with cycle type $(1,a, \ldots, a)$. Since $f$ is bijective, every equation $f^r(x)\approx f^s(x)$, with $r\geq s$ can be written as $f^{r-s}(x)\approx x$. If $a$ divides $r-s$, then $\Pr(f^{r-s}(x)\approx x \mid \A)=1$, while if it does not divide $\Pr(f^{r-s}(x)\approx x \mid \A)=1/n$ or equals $1$, depending on whether we have a single fixed point or all the elements are fixed points, that is, $f=\id$.

We now prove the second statement. Let $(A,f)$ and $(A,g)$ be two isocyclic algebras and let $h:A\longrightarrow A$ be an isomorphism. Then $h(f(x))=g(h(x))$ for all $x\in A$. That is, $f$ and $g$ are conjugate permutations, which is equivalent to saying that $f$ and $g$ have the same type of cycle, see \cite{Humphreys1996} for a proof of this elementary fact.
\end{proof}

We introduce some terminology and notation. All groupoids with constant operation of order $n$ are isomorphic. Let's denote them by $\C_n$. We will say that a groupoid is \emph{right isocyclic} (left isocyclic) if $xy=f(x)$ ($xy=f(y)$) where $f$ is an isocyclic permutation. We will say that a groupoid is isocyclic if it is either right or left isocyclic.
Two left isocyclic groupoids, with operations $x*y=f(x)$ and $x*_g y=g(y)$, are isomorphic if and only if $f$ and $g$ are conjugate permutations. Therefore, to characterize a left isocyclic groupoid, it suffices to specify the cycle type $(1,a,\ldots, a)$. Since the order $n$ of the groupoid must satisfy $n=sa+1$ where $s$ is the number of cycles, a left isocyclic groupoid is determined by the order and length of the cycle $a$. We will denote this groupoid by $\I_n^{a-}$. The same applies to right isocyclic groupoids, which we will denote by $\I_n^{a+}$.

\begin{theorem} \label{TheoClasGen} If a groupoid of order $n$ has minimal spectrum, then it is isomorphic to $\C_n$, $\I_n^{a-}$ or $\I_n^{a+}$, for some divisor $a$ of $n-1$, or it is a quasigroup.
\end{theorem}

\begin{proof} We begin by considering the equation $xy\approx xz$. We have that
$$\faktor{(xy\approx xz)}{(y\approx z)}=(xy\approx xy),$$
which is an identity that holds in any algebra. Therefore, by Lemma~\ref{LemmaIdentEq} an algebra $\A$ with minimal spectrum will satisfy one of the two formulas
\begin{enumerate}[(i)]
\item $\forall x \forall y \quad xy\approx xz$,
\item $\forall x \forall y \quad xy\approx xz \rightarrow y\approx z$.
\end{enumerate}
If we consider the symmetric equation $yx\approx zx$, $\A$ will satisfy one of the two formulas
\begin{enumerate}
\item[(i)'] $\forall x \forall y \quad yx\approx zx$,
\item[(ii)'] $\forall x \forall y \quad yx\approx zx \rightarrow y\approx z$.
\end{enumerate}
We consider all possible cases.
\begin{enumerate}
\item (i) with (i)'. Since $xy\approx xz$ for any $y$ and $z$, we have that there exists a function $f: A\longrightarrow A$ such that $xy=f(x)$ for all $y$. By symmetry, there exists a function $g: A \longrightarrow A$ such that $xy=g(y)$. Then, for all $x,y$, $f(x)=g(y)$. If we fix a $y_0$, then for all $x$, $f(x)=g(y_0)$. Therefore, $f$ is constant, and the groupoid is isomorphic to $\C_n$.
\item (i) with (ii)'. Following the previous notation, we write $xy=f(x)$. The condition $(ii)'$ says that the groupoid is right cancellative. So, if we have that $yx=zx$, then $y=z$. But $yx=zx$ is equivalent to $f(y)=f(z)$. This tells us that $f$ is injective, and therefore bijective. Now it suffices to consider equations of the form
$$\underbrace{(((\cdots )x)x)x}_{m \mbox{ \footnotesize times}}\approx x \iff f^m(x)=x,$$
and by means of the same arguments as in Theorem~\ref{TheoMonoUnary}, we obtain that $f$ is an isocyclic permutation and therefore the groupoid is isomorphic to $\I_n^{a-}$, for a certain $a$.

\item (ii) with (i)'. This is the symmetric case of the previous point and we conclude that the groupoid is isomorphic to
$\I_n^{a+}$.

\item (ii) with (ii)'. The groupoid is left and right cancellative, that is, $\A$ is a quasigroup.
\end{enumerate}
\end{proof}

We now deduce further properties.
Recall that a groupoid with operation $*$ is said to be derived from a groupoid with operation $\cdot$ if both have the same universe and $x*y=t(x,y)$, where $t$ is a term defined by the operation $\cdot$. Suppose the groupoid $\A$ is derived from the groupoid $\B$, of order $n$. If $\B$ has minimal spectrum,
$$\left\{\frac{1}{n},1\right\} \subseteq \PSpec(\A) \subseteq \PSpec (\B) =\left\{\frac{1}{n},1\right\},$$
and therefore $\A$ also has minimal spectrum.
\begin{proposition} \label{PropoProp} Every groupoid $\A$ with minimal spectrum satisfies the following properties.
\begin{enumerate}[(i)]
\item Either it is commutative, or it is anticommutative.
\item Either the diagonal of the operation table is constant, or if the diagonal contains two idempotent elements, then all of $\A$ is idempotent.
\item Every single variable term that is not constant is an isocyclic permutation.
\end{enumerate}
\end{proposition}
\begin{proof} \begin{enumerate}[(i)]
\item
We start from
$$\faktor{xy \approx yx}{x\approx y}=xy\approx xy.$$
The last identity is always true. Therefore, by Lemma~\ref{LemmaIdentEq}, it is universally satisfied that either $xy\approx yx$ or $xy\approx yx \rightarrow x\approx y$.
\item This follows from (iii). Define $f(x)=xx$. If $f$ is not constant, then $f$ is an isocyclic permutation, and this forces that if there are two idempotent elements, all elements must be idempotent.
\item Consider the monounary algebra $(A,f)$ where $f$ is the operation defined by a term with a single variable. If $\A$ has minimal spectrum, the monounary algebra also has it, and we have already seen that then either $f$ is constant or it is an isocyclic permutation.
\end{enumerate}
\end{proof}

\section{Minimal spectrum semigroups}

We do not know a complete characterization of groupoids with minimal spectrum. However, if we restrict ourselves to semigroups, a complete characterization can be obtained. Recall that for any algebra $\PSpec(\A^m)=\{ \alpha^m \mid \alpha \in \PSpec(\A)\}$. Hence $\A$ has minimal spectrum if and only if $\A^m$ does. We briefly verify this. If $\A$ has minimal spectrum of order $n$, then
$$\PSpec(\A^m)=\{ \alpha^m \mid \alpha \in \{1/n, 1\}\}=\{1, 1/n^m\}.$$
Conversely, if $\A^m$ has minimal spectrum, then so does $\A$.

\begin{lemma} \label{LemmaGrupAb1} Let $a$ and $n$ be integers such that $0<a<n$, with $\gcd(a,n)\neq 1$. Then,
$$\frac{1}{n}< \Pr( ax\approx 0 \mid \mathbb{Z}_n)<1.$$
\end{lemma}
\begin{proof} It suffices to prove that
$$1 < |\{ ax \approx 0\}_{\mathbb{Z}_n}|<n.$$
First suppose that $|\{ ax \approx 0\}_{\mathbb{Z}_n}|=n$. This is equivalent to $\{ ax \approx 0\}_{\mathbb{Z}_n}=\mathbb{Z}_n$. That is, for any $x\in \mathbb{Z}_n$, $ax=0$. This occurs only when $a=0$, which is not possible by hypothesis.

Now suppose that $|\{ ax \approx 0\}_{\mathbb{Z}_n}|=1$. This is equivalent to $\{ ax \approx 0\}_{\mathbb{Z}_n}=\{0\}$. However, there exists another nontrivial solution. Let $d=\gcd(a,n)$ and let $x=\frac{n}{d}$. Then,
$$ax=a \left(\frac{n}{d}\right)= \left(\frac{a}{d} \right) n \equiv 0 \pmod n.$$
Since $d \neq 1$ by hypothesis, the solution is not trivial.
\end{proof}

\begin{lemma} \label{LemmaGrupAb2} If the group $\mathbb{Z}_{n_1} \oplus \cdots \oplus \mathbb{Z}_{n_s}$ has minimal spectrum, then $n_1=\cdots=n_s$.
\end{lemma}
\begin{proof} Since the groups considered here are finite, $n_1, \ldots, n_s>1$, and at least two of them are distinct. Without loss of generality, we can assume that
$$n_1=\cdots =n_r < n_{r+1} \leq \cdots \leq n_s,$$
for some $r<s$. Let $a=n_1$ and let $P$ be the probability
$$P=\Pr(ax\approx 0 \mid \mathbb{Z}_{n_1} \oplus \cdots \oplus \mathbb{Z}_{n_s})=\prod_{i=1}^s \Pr(ax\approx 0 \mid \mathbb{Z}_{n_i}).$$
Since $a=n_1$, $\Pr(ax\approx 0 \mid \mathbb{Z}_{n_i})=1$, for all $i=1, \ldots, r$. We can  rewrite $P$ as
$$P=\prod_{i=r+1}^s \Pr(ax\approx 0 \mid \mathbb{Z}_{n_i}).$$
We now consider the following partition:
\begin{align*}
I&=\{ i \in \{r+1, \ldots, s\} \mid \gcd(a,n_i)=1\}, \\ J&=\{ j \in \{r+1, \ldots, s\} \mid \gcd(a,n_j)\neq 1\}.
\end{align*}
Note that since we are assuming $r<s$, $I\cup J\not=\emptyset$. Therefore,
$$P=\prod_{i \in I} \Pr(ax\approx 0 \mid \mathbb{Z}_{n_i}) \cdot \prod_{j \in J} \Pr(ax\approx 0 \mid \mathbb{Z}_{n_j}).$$
The first product is easy to calculate:
$$P=\prod_{i \in I} \frac{1}{n_i} \cdot \prod_{j \in J} \Pr(ax\approx 0 \mid \mathbb{Z}_{n_j}).$$
For the second product we use the previous Lemma~\ref{LemmaGrupAb1} for each $j$, that is,
$$\frac{1}{n_j}< \Pr( ax\approx 0 \mid \mathbb{Z}_{n_j})<1.$$
Applying this to $P$ yields
$$\prod_{i \in I} \frac{1}{n_i} \prod_{j\in J} \frac{1}{n_j}<P<\prod_{i \in I} \frac{1}{n_i}.$$
Finally, if $r<s$, observe that
$$\frac{1}{n_1\cdots n_s}< \prod_{i=r+1}^s \frac{1}{n_i}=\prod_{i \in I} \frac{1}{n_i} \prod_{j \in J} \frac{1}{n_j}<P<\prod_{i \in I} \frac{1}{n_i}<1,$$
where the first inequality arises from the fact that $I\cup J\not=\emptyset$. But these inequalities show that $\mathbb{Z}_{n_1} \oplus \cdots \oplus \mathbb{Z}_{n_s}$ is not minimal.
\end{proof}

\begin{lemma}\label{LemmaGrupAb3} If $p$ is a prime number,
$$\PSpec(\mathbb{Z}_{p^r})=\left\{ p^{s-r} \mid 0\leq s\leq r \right\}.$$
\end{lemma}
\begin{proof} Consider the nontrivial equation $\sum_{i=1}^k a_ix_i =0$ and let $d=\gcd(a_1, \ldots, a_k)$. Setting $b_i=a_i/d$, we obtain the equation
$$d \sum_{i=1}^k b_ix_i \equiv 0 \pmod {p^r}.$$
Since $p$ is prime, $\gcd(d, p^r)=p^s$, for some $0\leq s \leq r$. Thus, $d=d' p^s$, where $\gcd(d', p^s)=\gcd(d',p^r)=1$. Hence we may cancel the factor $d'$, and the equation is equivalent to
\[p^s \sum_{i=1}^k b_ix_i \equiv 0 \pmod {p^r}. \tag{$\ast$}\]
Now consider the equation $p^s y \equiv 0 \mod p^r$. It has $p^s$ solutions, which are of the form $y=jp^{r-s}$, for $j=1, ..., p^s$. Therefore, the solution set of $(\ast)$ decomposes into the disjoint union of solutions of $p^s$; one equation for each $j$:
\[\sum_{i=1}^k b_ix_i \equiv jp^{r-s} \pmod {p^r}. \tag{$\ast\ast$} \]
Observe that there exists some $b_i$ such that $\gcd(b_i, p^r)=1$. Without loss of generality, assume that $\gcd(b_k, p^r)=1$. Therefore, $b_k$ is an invertible element in the ring $\mathbb{Z}_{p^r}$. The solution set of ($\ast\ast$) is given by
$$\left\{ (x_1, \ldots, x_{k-1}, b_k^{-1} \left(-jp^{r-s}+\sum_{i=1}^{k-1}b_ix_i \right) \midd x_1, \ldots , x_{k-1}\in \mathbb{Z}_{p^r}\right\},$$
which has cardinality equal to $p^{r(k-1)}$. Since $j$ ranges from $j=1, \ldots, p^s$, the number of solutions of ($\ast$) is $p^s \cdot p^{r(k-1)}=p^{s+r(k-1)}$, and dividing by $p^{rk}$ we obtain the probability
$$\Pr\left(\sum_{i=1}^k a_ix_i \approx 0 \midd \mathbb{Z}_{p^r}\right)=p^{s-r}.$$
Finally, since $0\leq s \leq r$, the result follows.
\end{proof}

\begin{theorem} The semigroups of order $n$ with minimal spectrum are exactly $\C_n$, $\I_n^{1+}$, $\I_n^{1-}$ and $(\mathbb{Z}_p)^m$ with $p$ prime and $n=p^m$.
\end{theorem}

\begin{proof} By Theorem~\ref{TheoClasGen}, either the semigroup is constant, or it is isocyclic, or it is a quasigroup. In the first case, we are already done. Suppose we have a left isocyclic groupoid (the right case is symmetric). If the operation $x*y=f(x)$ is associative, then $x(yz)=xf(y)=f(x)$ must be equal to $(xy)z=f(x)z=f^2(x)$. However, the identity is the only permutation $f$ such that $f^2=f$. Thus, the only left isocyclic semigroups are $\I_n^{1-}$ and $\I_n^{1+}$ for the right case.

We now consider the case of quasigroups. We know from elementary theory that every finite cancellative semigroup is a group, see \cite[pg. 61]{Howie1995}. We thus consider a group $\mathcal{G}$ with minimal spectrum. By Proposition~\ref{PropoProp}, it is either commutative or anticommutative. However, no nontrivial group can be anticommutative. For example, the pair of elements $1$ and $g\not=1$ always commute. Therefore, we consider only abelian groups. By the classification theorem of abelian groups, there exist prime numbers $p_1, \ldots, p_m$ and positive exponents $r_1, \ldots, r_m$ such that
$$\mathcal{G}\cong \mathbb{Z}_{{p_1}^{r_1}} \oplus \cdots \oplus \mathbb{Z}_{{p_m}^{r_m}}.$$
However, by Lemma~\ref{LemmaGrupAb2}, if $\mathcal{G}$ is minimal, $p_1^{r_1}=\cdots =p_s^{r_s}=p^r$ and therefore
$$\mathcal{G}\cong \left( \mathbb{Z}_{{p}^{r}} \right)^m.$$
The power of an algebra has minimal spectrum if and only if the algebra has minimal spectrum. It remains to analyze the case $\mathbb{Z}_{p^r}$. By Lemma~\ref{LemmaGrupAb3}, this group has minimal spectrum if and only if $r=1$.
\end{proof}

\section{Weak associativity and other relaxed identities}
Since associativity forces every quasigroup to be a group, a case we have already studied, we now consider weaker forms of associativity. In the theory of quasigroups (see \cite{Pflugfelder1990, Bruck1971}), the following identities are important specializations of associativity.
\begin{align*}
x(xy)\approx (xx)y & \qquad \mbox{(left-alternative law)}\\
(yx)x\approx y(xx) & \qquad \mbox{(right-alternative law)}\\
x(yx)\approx (xy)x &\qquad \mbox{(flexible law)}
\end{align*}
We refer to these identities as \emph{weak associativity laws}. Under commutativity, the left and right alternative identities are equivalent. A groupoid is said to be \emph{alternative} if it is left and right alternative, and \emph{$2/3$-associative} if it satisfies any two of the three laws.

\begin{theorem} \label{TheoLoop} If a loop has minimal spectrum, then it is a group. 
\end{theorem}
\begin{proof} Let $\Gamma$ denote the set of triples for which associativity holds, 
$$\Gamma=\{ (xy)z\approx x(yz) \}_\G= \{ (x,y,z) \in G^3 \mid  (xy)z= x(yz) \}.$$
We assume that $n>1$, since the case $n=1$ is trivial.
We give a lower bound on the number of associative triples. If $\G$ is a loop, there is a neutral element $1$. Define three sets,
$$\Gamma_1=\{(1,y,z) \}, \quad \Gamma_2=\{(x,1,z)\}, \quad \Gamma_3=\{(x,y,1) \},$$
 where $x,y,z \in \G$.
Note that $\Gamma_1\cup \Gamma_2\cup \Gamma_3 \subseteq \Gamma$. 
We compute the cardinality using the inclusion–exclusion principle: 
\begin{align*}
|\Gamma_1\cup \Gamma_2\cup \Gamma_3|&=|\Gamma_1|+|\Gamma_2|+|\Gamma_3|\\
&-(|\Gamma_1\cap \Gamma_2|+|\Gamma_2\cap \Gamma_3|+|\Gamma_1\cap \Gamma_3|)+|\Gamma_1\cap \Gamma_2\cap \Gamma_3|\\
&=n^2+n^2+n^2-(n+n+n)+1=3n^2-3n+1
\end{align*}
We then have that, if $n>1$,
$$\Pr((xy)z\approx x(yz)\mid \G)=\frac{|\Gamma|}{n^3}\geq \frac{|\Gamma_1\cup \Gamma_2\cup \Gamma_3|}{n^3}=\frac{3n^2-3n+1}{n^3}>\frac{1}{n}, $$
since $3n^2-3n+1>n^2$ for $n>1$.
Since $\G$ has minimal spectrum, necessarily $\Pr((xy)z\approx x(yz)\mid \G)=1$ and $\G$ is associative.
\end{proof}

\begin{theorem} If a quasigroup with minimal spectrum is $2/3$-associative, then it is a group.
\end{theorem}
\begin{proof}  Suppose first that $\G$ is an alternative groupoid. We count associative triples in $\Gamma$ as in  the proof of Theorem~\ref{TheoLoop}. Define the sets
$$\Gamma_1=\{(x,x,z) \}, \quad \Gamma_2=\{(x,y,y) \},$$
 where $x,y,z \in \G$. The triples in $\Gamma_1$ satisfy the left alternative identity, while the triples in $\Gamma_2$ satisfy the right alternative identity. Therefore, $\Gamma_1 \cup \Gamma_2 \subseteq \Gamma$. We compute the cardinality:
\begin{align*}
|\Gamma_1\cup \Gamma_2|=|\Gamma_1|+|\Gamma_2|-|\Gamma_1\cap \Gamma_2|
=n^2+n^2-n=2n^2-n.
\end{align*}
If $n>1$,
$$\Pr((xy)z\approx x(yz)\mid \G)=\frac{|\Gamma|}{n^3}\geq \frac{|\Gamma_1\cup \Gamma_2|}{n^3}=\frac{2n^2-n}{n^3}>\frac{1}{n}.$$
Hence associativity follows.

Next we show that if $\G$ is left alternative and flexible, then it is associative. We define the sets
$$\Gamma_1=\{(x,x,z) \}, \quad \Gamma_2=\{(x,y,x) \}.$$
The first set contains associative triples that satisfy left alternativity and the second contains associative triples that satisfy flexibility. The enumeration yields the same bound as above. The case of right alternativity plus flexibility is symmetric.
\end{proof}

We conclude with a general structural lemma extending the previous arguments.
 We will say that two equations are equivalent in $\A$ if the set of their solutions is the same. We will say that two equations $t_1\approx t_1'$ and $t_2\approx t_2'$ are \emph{incomparable} in $\A$ if the sets of their solutions are not comparable (neither contains the other).

\begin{lemma} \label{Lemma2Eq} Let $\A$ be an algebra with minimal spectrum and let $t\approx t'$ be an equation. If $\A$ satisfies the specializations,
$$\faktor{t\approx t'}{x\approx y} \qquad \faktor{t\approx t'}{x'\approx y'},$$
and these are incomparable, then $\A$ satisfies the identity $t\approx t'$.
\end{lemma}
\begin{proof} Let $x,y,x',y'$ be variables of the equation $t\approx t'$. We define the sets
$$\Gamma_1=\{ \vec{z}\in A^k \mid x=y \} \quad \mbox{ and } \quad \Gamma_2=\{ \vec{z}\in A^k \mid x'=y' \}.$$
By hypothesis, $\A$ satisfies both specializations. This means that if we identify the variables $x=y$ inside the $k$-tuple $\vec{z}$, then $t(\vec{z})=t'(\vec{z})$. Therefore, $\Gamma_1 \subseteq \{ t\approx t'\}_\A$. And the same holds for the second specialization. Furthermore, $|\Gamma_1|=|\Gamma_2|= n^{k-1}$.
On the other hand, if the equations are incomparable, the two identifications correspond to pairs of different variables, $\{x,y\}\neq \{x',y'\}$, impose two independent restrictions, and we have that
$|\Gamma_1 \cap \Gamma_2| = n^{k-2}$.
Therefore,
$$|\Gamma_1 \cup \Gamma_2| =|\Gamma_1|+|\Gamma_2|-|\Gamma_1\cap \Gamma_2|= 2n^{k-1}-n^{k-2}$$
Finally, since $\Gamma_1 \cup \Gamma_2 \subseteq \{ t\approx t'\}_\A$, dividing by $n^k$ we have
$$\Pr(t \approx t'\mid \A)=\frac{|\{t\approx t'\}_{\A}| }{n^k}\geq \frac{|\Gamma_1 \cup \Gamma_2|}{n^k} \geq \frac{2n^{k-1}-n^{k-2}}{n^k}=\frac{2}{n}-\frac{1}{n^2}>\frac{1}{n}.$$
Since $\A$ has minimal spectrum, $\Pr(t \approx t'\mid \A)=1$.
\end{proof}

Let us see some applications of this Lemma. First, note that the incomparability condition is necessary. An example is the medial identity
$$(xy)(uv) \approx (xu)(yv).$$
When we identify any two variables, we always obtain the same equivalent identity:
$$(xx)(uv)\approx (xu)(xv).$$
Although the syntactic expression may vary, the set of solutions is always the same. Thus, the lemma does not apply to the medial identity. 
 Nonetheless, we saw that the associative identity gives three different and incomparable specializations.
 As another example, recall that a loop satisfying the Moufang identity is a group. Consider a quasigroup that is Moufang without having a neutral element. The left Moufang identity is
$$x(y(xz))\approx((xy)x)z,$$
which gives rise to three possible incomparable specializations,
$$x(x(xz))\approx ((xx)x)z, \qquad x(y(xy))\approx ((xy)x)y, \qquad x(y(xx))\approx ((xy)x)x.$$
By applying Lemma~\ref{Lemma2Eq} we obtain that if a groupoid with minimal spectrum satisfies two left Moufang specializations, then it is left Moufang.

These results show that minimality of the spectrum imposes strong rigidity on algebraic structures. In particular, it yields a complete classification in the case of semigroups, and forces loops and weak associativity conditions to collapse into associativity, and hence into group.
It remains an open question whether minimal spectrum is sufficient to classify finite quasigroups up to isomorphism.

\section*{Declarations}

\begin{itemize}
\item Funding  
Not applicable.

\item Conflict of interest/Competing interests.
The author declares that he has no conflict of interest.

\item Ethics approval and consent to participate.
Not applicable.

\item Consent for publication.
Not applicable.

\item Data availability.
No data were generated or analyzed in this study.

\item Materials availability. 
Not applicable.

\item Code availability.
Not applicable.

\item Author contribution.
The author is solely responsible for the content of this article.
\end{itemize}

\bibliography{BibPspec}

\end{document}